\documentclass{svmult}
\usepackage{mathptmx,amsfonts}
\begin{document}
\title*{Upper semicontinuity of KLV polynomials for certain blocks of Harish-Chandra modules}
\author{William M. McGovern}
\titlerunning{Upper semicontinuity}
\institute{William M. McGovern \at University of Washington, Box 354350, Seattle, WA  98195\email{mcgovern@math.washington.edu}}
\maketitle

\abstract{We show that the coefficients of Kazhdan-Lusztig-Vogan polynomials attached to certain blocks of Harish-Chandra modules satisfy a monotonicity property relative to the closure order on $K$-orbits in the flag variety.}

\section*{}

Let $G$ be a complex connected reductive group with Lie algebra $\mathfrak g$ and Borel subgroup $B$.  Recall that the flag variety $G/B$ decomposes into finitely many $B$-orbits $\mathcal O_w$, which are indexed by elements $w$ of the Weyl group $W$ of $G$.  The closure $\bar{\mathcal O_w}$ of an orbit $\mathcal O_w$ is called a Schubert variety.  Kazhdan and Lusztig introduced polynomials $P_{v,w}$ in one variable $q$, indexed by pairs $v,w$ of elements in $W$ in \cite{KL79}, which they later showed measured the singularities of Schubert varieties \cite{KL80}.  More precisely, they showed that the coefficient $c_i$ of $q^i$ in $P_{v, w}$ satisfies 
$$
c_i = \dim IH^{2i}_x(\bar{\mathcal O_w};\bar{\mathbb Q}_p)
$$
\noindent for any $x\in\mathcal O_v$, where the right side denotes the local $2i$th intersection cohomology group with values in the constant sheaf $\mathbb Q_p, p$ is a prime and $w_0$ is the longest element of $W$ \cite{KL80}.
A fundamental result of Irving, first proved in \cite{I88}, using results of \cite{GJ81}, asserts that the singularities of $\bar{\mathcal O}_w$ increase as one goes down; more precisely, if $c^{v,w}_i$ denotes the coefficient of $q_i$ in $P_{v,w}$, then $c^{v,w}_i\ge c^{v',w}_i$ whenever $v\le v'\le w$ in the Bruhat order on $W$.  (Following Li and Yong \cite{LY10}, we call this property upper semicontinuity.)  Irving's proof uses representation theory; later Braden and MacPherson gave a geometric argument in \cite{BM01}.  Braden and MacPherson's proof applies to many stratified varieties with a torus action (not just Schubert varieties), but not to closures of orbits of a symmetric subgroup on $G/B$.  The purpose of this note is to establish the corresponding inequality in some cases for coefficients of Kazhdan-Lusztig-Vogan (KLV) polynomials such orbit closures.

So let $\theta$ be an involutive automorphism of $G$ and $K$ the fixed points of $\theta$ acting on $G$.  If $\mathcal O,\mathcal O'$ are two $K$-orbits in $G/B$ with $\bar{\mathcal O}\subset\bar{\mathcal O'}$ and if $\gamma,\gamma'$ are $K$-equivariant sheaves of one-dimensional $\bar{\mathbb Q}_p$ vector spaces on $\mathcal O,\mathcal O'$, respectively (more briefly, one-dimensional sheaves), then Lusztig and Vogan have constructed a polynomial $P_{\gamma,\gamma'}$ such that the coefficient $d_i$ of $q^i$ in $P_{\gamma,\gamma'}$ equals the dimension of the local $2i$th intersection cohomology sheaf of the Deligne-Goresky-MacPherson extension of $\gamma'$ to the closure of $\mathcal O'$, supported at a point in $\mathcal O$ \cite{LV83,V83}.  If $\gamma$ and $\gamma'$ are trivial then we write $P_{\mathcal O,\mathcal O'}$ instead of $P_{\gamma,\gamma'}$.  We then ask under what conditions, given three orbits $\mathcal O_1,\mathcal O_2,\mathcal O_3$ with $\bar{\mathcal O_1}\subset\bar{\mathcal O_2}\subset\bar{\mathcal O_3}$, do we have
$$
d^{\mathcal O_1,\mathcal O_3}_i\ge d^{\mathcal O_2,\mathcal O_3}_i
$$
\noindent for all $i$, where the terms denote the coefficients of $q^i$ in the polynomials corresponding to the pair of orbits in the superscripts.  In general, this fails, for example if $G=SL(3)$ (as we will observe in an example below).  The problem stems from the existence of nontrivial sheaves $\gamma$, even though the inequality is stated for trivial $\gamma$ only.  Now we can state our result.

\begin{theorem}
With notation as above, assume that all $K$-orbits $\mathcal O$ admit only the trivial sheaf (equivalently, all orbits $\mathcal O$ are simply connected, or (as is well known) all Cartan subgroups of the real form $G_0$ of $G$ corresponding to $K$ are connected).  If $\bar{\mathcal O}_1\subseteq\bar{\mathcal O}_2\subseteq\bar{\mathcal O}_3$ then $d^{\mathcal O_1,\mathcal O_3}_i\ge d^{\mathcal O_2,\mathcal O_3}_i$.
\end{theorem}

\begin{proof}
 The proof follows Irving's proof in \cite{I88} for Schubert varieties closely, supplemented by basic facts on block equivalence for Harish-Chandra modules from \cite{V83} and \cite{V81}.  We begin by observing that the hypothesis on $G$ and $K$ implies that none of the roots of $\frak g$ relative to any $K$-orbit (or to the corresponding $\theta$-stable Cartan subgroup of $G$) are of type II, whence the recursion formulas of \cite[6.14,(a,b,c,e)]{V83} imply that the constant term of $P_{\mathcal O,\mathcal O'}$ equals 1 whenever $\bar{\mathcal O}\subseteq\bar{\mathcal O'}$.  Moreover, applying the circle action of \cite[\S5]{V83}, we find that whenever any orbit lying in the image of a simple reflection applied to a given orbit has closure containing that of the orbit, then this image is single-valued and the simple root in question is either complex or noncompact imaginary type I.  As every orbit can be obtained from a closed orbit by repeated application of simple reflections raising it in the closure order, it follows that there is a single block $\mathcal B$ containing all simple $(\frak g,K)$ modules of trivial infinitesimal character.
 
Now we appeal to ``IC4  duality''.  Vogan has shown in \cite{V82} that there is another (possibly disconnected) complex group $G'$ with symmetric subgroup $K'$ and Lie algebra $\mathfrak g'$ and a block $\mathcal B'$ of $(\mathfrak g',K')$-modules, which we may take to have trivial infinitesimal character, such that there is a bijection $D$ between the set $S$ of $K$-orbits in $G/B$ and a set $S'$ of one-dimensional sheaves $\gamma$ over $K'$-orbits in $G'/B'$ that is order-reversing on the underlying orbits relative to the closure order.  (The sheaves in $S'$ parametrize the irreducible modules in $\mathcal B'$.)  In particular, since there is a unique maximal (open) $K$-orbit in $G/B$, there is a unique orbit minimal among the orbits corresponding  to the sheaves in $S'$.  Moreover, the values of the KLV polynomials for $\mathcal B$ at 1 count the multiplicities of composition factors in standard $(\mathfrak g',K')$-modules for $\mathcal B'$; both the irreducible and the standard modules in $\mathcal B'$ are indexed by elements of $S'$.  Casian and Collingwood have refined this result along the lines of the Gabber-Joseph refinement of the Kazhdan-Lusztig conjecture for Verma modules:  they showed that the coefficients of KLV polynomials count multiplicities of composition factors in the layers of the weight filtrations of standard modules in $\mathcal B'$ \cite[1.12]{CC89}, with standard modules indexed by sheaves on lower orbits occurring further down than those indexed by sheaves on higher orbits.  (Here the standard modules in $\mathcal B'$ are normalized to have unique irreducible quotients, not unique irreducible submodules.)  In particular, all standard modules in $\mathcal B'$ have a single copy of the unique simple standard module in this block at the lowest layer of the weight filtration.  Now the proofs of Theorem 2.6.3 and Lemma 2.6.5 of \cite{CI92} carry over to show that the simple standard module is the unique simple subquotient of any standard module in $\mathcal B'$ of largest GK dimension and is the socle of that module.  It then follows from \cite{C89} that all standard modules in $\mathcal B'$ are rigid, so that their socle and weight filtrations coincide.  

We now show inductively that whenever $\gamma,\gamma'$ are two elements of $S'$ with $\gamma=D(\mathcal O_1),\gamma'=D(\mathcal O_2)$ and $\bar{\mathcal O}_2\subseteq\bar{\mathcal O}_1$ , then the standard module $X_\gamma$ indexed by $\gamma$ embeds in $X_{\gamma'}$.  (This condition on $\gamma,\gamma'$ is \textsl{not} equivalent to requiring that $\gamma\le\gamma'$ in the Bruhat order.) We have just shown that this embedding holds if $X_\gamma$ is simple.  In general we assume inductively that this result holds for all pairs $\gamma_1,\gamma_2$ with either $\gamma_2<\gamma'$ in the closure order, or  else $\gamma_2=\gamma',\gamma_1<\gamma$.  As in the proof of Proposition 6.14 of \cite{V83}, locate a simple reflection $s=s_\alpha$ such that either $\gamma,\gamma'$, or both may be realized as the single-valued image of $s$ applied via the circle action of $W$ to an appropriate element of $S$.  For definiteness assume that $\gamma=s\circ\gamma_1$ while $\alpha$ is real for $\gamma'$ and does not satisfy the parity condition, so that $s\circ\gamma'$ is empty; the other cases are similar.  Applying the wall-crossing operator corresponding to $s$ to the standard module $X_{\gamma_1}$, we get a module having $X_{\gamma_1}$ as a submodule with quotient specified by either Proposition 8.2.7 or Proposition 8.4.5 of \cite{V81}, depending on whether $\alpha$ is complex or noncompact imaginary type I for $\gamma_1$.  Applying the same operator to $X_{\gamma'}$, we get a module having $X_{\gamma'}$ as a submodule whose quotient is again $X_\gamma'$, by Proposition 8.4.9 of \cite{V81}.  Hence the embedding of $X_{\gamma_1}$ into $X_{\gamma'}$ induces an embedding of the image $T_\alpha(X_{\gamma_1})$ of $X_{\gamma_1}$ under the wall-crossing operator $T_\alpha$ into $T_\alpha(X_{\gamma'})$, which induces a map from the quotient of $T_\alpha(X_{\gamma_1})$ by its submodule isomorphic to $X_{\gamma_1}$ to the corresponding quotient of $T_\alpha(X_{\gamma'})$, which is $X_{\gamma'}$.   This induced map between the quotients is injective when restricted to the socle of its domain, since there is exactly one copy of the unique irreducible standard module in $\mathcal B'$ inside $X_{\gamma_1}$, at the lowest level of the socle filtration, and exactly two copies of the irreducible standard module in $T_\alpha(X_{\gamma_1})$, one at the lowest and the other at the next-to-lowest level of the socle filtration.  Hence the induced map is injective, and we get an embedding of $X_\gamma$ into $X_{\gamma'}$, as desired.

We now conclude the proof in exactly the same way that Irving did in the Schubert variety case \cite[Corollary 4]{I88}:  $d^{\mathcal O_1,\mathcal O_3}_i$ counts the multiplicity of a suitable composition factor in the $2i$th level of the socle filtration of a suitable standard module $X$ for $(\mathfrak g',K')$, while $d^{\mathcal O_2,\mathcal O_3}_i$ counts the multiplicity of the same composition factor in the $2i$th level of the socle filtration of a submodule of $X$.  The desired inequality follows at once from the definition of the socle filtration.  Notice that the Bruhat order of \cite{V83} coincides with the order by containment of closures for $K$-orbits in $G/B$, in contrast to the situation for $K'$-orbits in $G'/B'$, thanks to \cite[7.11,(vii)]{RS90}.
\end{proof}

We remark that the theorem extends to KLV polynomials for principal blocks (containing the trivial representation) of $(\mathfrak g,K)$-modules even if not all Cartan subgroups of the real form $G_0$ of $G$ are connected, provided that there is only a single conjugacy class of disconnected Cartan subgroups, the groups in this class have only two components, and all modules attached to trivial sheaves on orbits lie in the principal block.  This covers the cases $G=GL(2p), K = GL(p) \times GL(p);  G = SO(2n), K = GL(n);  G = E_7,  K = E_6 \times \mathbb C$.  On the other hand, the case $G = G_2, K = SL(2)\times  SL(2)$ does not work:  there is only one conjugacy class of disconnected Cartan subgroups, but the groups in it have four components.  Here the theorem holds for the nonprincipal block, which has only one simple module, but fails for the principal one.  A more subtle failure occurs for $G=SL(3),K=SO(3)$:  here there is only conjugacy class of disconnected Cartan subgroups and the groups in it have only two components, but not all modules attached to trivial sheaves on orbits lie in the same block.  There is a nontrivial sheaf $\gamma$ attached to the open orbit  $\mathcal O_1$ in $G/B$ and a lower orbit $\mathcal O_2$ (admitting only the trivial sheaf), such that the KLV polynomial attached to the pair $((\mathcal O_1,\gamma),\mathcal O_2)$ s 0, but the one attached to $((\mathcal O_1,\gamma),(\mathcal O_1,\gamma))$ is 1.

We further remark that Collingwood and Irving have explored the properties of Harish-Chandra modules in the block $\mathcal B$ (rather than its dual block $\mathcal B'$), in the special case where the real form $G_0$ has only one conjugacy class of Cartan subgroups.  Here the standard modules do not satisfy inclusion relations corresponding to inclusion of orbit closures, but many other familiar properties of modules in category $\mathcal O$ do carry over \cite{CI92}.

\begin{acknowledgement}
I would like to thank Ben Wyser and Alex Yong for bringing this problem to my attention, Peter Trapa for helpful messages, and the referee for many thoughtful comments.
\end{acknowledgement}

\end{document}